\documentclass{article}

\usepackage{amsfonts}
\usepackage{amsmath}
\usepackage{amssymb}
\usepackage{wasysym}
\usepackage{yfonts}  
\usepackage{amsthm}
\usepackage{amscd}   
\usepackage[all]{xypic}
\usepackage{mathrsfs}
\usepackage{setspace}

\CompileMatrices

\newtheorem{theorem}{Theorem}[section]
\newtheorem{lemma}[theorem]{Lemma}
\newtheorem{proposition}[theorem]{Proposition}
\newtheorem{corollary}[theorem]{Corollary} 
\newtheorem*{definition}{Definition}
\newtheorem{example}[theorem]{Example}

\newtheorem*{notation}{Notation and terminology}
\newtheorem*{acknowledgement}{Acknowledgements}
\newtheorem*{thmA}{Theorem A}
\newtheorem*{thmB}{Theorem B}

\newtheorem*{propA}{Proposition A}

\numberwithin{equation}{section}

\date{}

\begin{document}

\title{Torsion cohomology for solvable groups of finite rank}

\author{\sc Karl Lorensen}

\maketitle

\vspace{-20pt}

\centerline{\footnotesize Department of Mathematics and Statistics}
\begin{spacing}{.5}

\centerline{\footnotesize Pennsylvania State University, Altoona College}

\centerline{\footnotesize Altoona, PA 16601}

\centerline{\footnotesize USA}
\vspace{5pt}

\centerline{\footnotesize E-mail address: {\tt kql3@psu.edu}}

\end{spacing}
\begin{abstract}  We define a class $\mathcal{U}$ of solvable groups of finite abelian section rank which includes all such groups that are virtually torsion-free as well as those that are finitely generated.  Assume that $G$ is a group in $\mathcal{U}$ and $A$ a $\mathbb ZG$-module. If $A$ is $\mathbb Z$-torsion-free and has finite $\mathbb Z$-rank,  we stipulate a condition on $A$ that guarantees that $H^n(G,A)$ and $H_n(G,A)$ must be finite for $n\geq 0$.  Moreover, if the underlying abelian group of $A$ is a \v{C}ernikov group,  we identify a similar condition on $A$ that ensures that $H^n(G,A)$ must be a \v{C}ernikov group for all $n\geq 0$. 
\vspace{10pt}

\end{abstract}
\vspace{20pt}

\section{Introduction}

This paper continues the study of the homology and cohomology of solvable groups of finite abelian section rank begun by P. H. Kropholler and the author in \cite{kroploren}. As in the previous paper, we will use the shorter term {\it solvable group of finite rank} to refer to such a group. 
For a virtually torsion-free solvable group $G$ of finite rank and a $\mathbb ZG$-module $A$ that is torsion-free and has finite rank as an abelian group, the article \cite{kroploren} identifies a condition on $A$ that ensures that $H^n(G,A)$ and $H_n(G,A)$ are finite. The condition stipulates that $A$ must not possess a nonzero submodule whose underlying abelian group has a spectrum contained in that of $G$.  We remind the reader that the {\it spectrum} of a group is the set of primes for which the group has a quasicyclic section. Moreover, if $\pi$ is a set of primes, then a group is said to be {\it $\pi$-spectral} if its spectrum is contained in $\pi$. 

Rather than just focusing on groups that are virtually torsion-free, the present paper considers the following, larger class of solvable groups of finite rank.

\begin{definition}{\rm  Let $\mathcal{U}$ be the class of solvable groups $G$ of finite rank such that $G$ contains a nilpotent normal subgroup $N$ with the following two properties:
\vspace{5pt}

\noindent (i) The quotient $Q=G/N$ is finitely generated and virtually abelian.
\vspace{5pt}

\noindent (ii) If $T$ is the torsion subgroup of $N$ and $\pi$ the spectrum of $G$, then there is a $G$-invariant, $N$-central series in $T$  of finite length in which each infinite factor, viewed as a $\mathbb ZQ$-module, can be expressed as a quotient of a $\mathbb ZQ$-module that is torsion-free and $\pi$-spectral as an abelian group.} 

\end{definition}

\noindent Prominent examples of groups that belong to $\mathcal{U}$ are all finitely generated solvable groups of finite rank (Proposition 3.9).

Our first aim is to establish the following generalization of the main result in \cite{kroploren}. 

\begin{thmA} Let $G$ be a group in $\mathcal{U}$ with spectrum $\pi$. Assume that $A$ is a $\mathbb ZG$-module whose underlying abelian group is torsion-free and has finite rank. Suppose further that $A$ does not have any nontrivial $\mathbb ZG$-submodules that are $\pi$-spectral as abelian groups. Then $H^n(G,A)$ and $H_n(G,A)$ are finite 
for all $n\geq 0$. 

\end{thmA}

To prove Theorem A, we employ the same reasoning that is used for the virtually torsion-free case in \cite{kroploren}. The argument is based on the Lyndon-Hochschild-Serre (LHS) spectral sequences for cohomology and homology for the group extension $1\rightarrow N\rightarrow G\rightarrow Q\rightarrow 1$, where $N$ and $Q$ are as related in the definition of $\mathcal{U}$.  Moreover, in order to study the initial pages of these spectral sequences, we apply [{\bf 3}, Proposition A], describing a property of the functors ${\rm Ext}^n_{\mathbb Z\Gamma}$ and ${\rm Tor}_n^{\mathbb Z\Gamma}$ when $\Gamma$ is an abelian group. The success of this approach hinges on the $\mathbb ZQ$-module structure of $H_nN$. What is required is that $H_nN$ belongs to $\mathcal{C}(Q,\pi)$, where $\mathcal{C}(Q,\pi)$ is the class of $\mathbb ZQ$-modules that can be constructed from modules that are virtually torsion-free of finite rank and $\pi$-spectral by forming finitely many quotients and extensions. The main theorem in \cite{kroploren} is established by showing that $H_nN$ belongs to $\mathcal{C}(Q,\pi)$ if $N$ is torsion-free.  In the present article, we invoke L. Breen's results from \cite{breen} concerning the integral homology of abelian groups in order to show that $H_nN$ always lies in $\mathcal{C}(Q,\pi)$ if $G$ belongs to $\mathcal{U}$ (Corollary 3.8). 

The second, and primary, purpose of this article is to examine the situation where the underlying additive group of $A$ is torsion; more specifically, we will suppose that $A$ is a {\it \v{C}ernikov group}, meaning that it is a finite extension of a direct sum of finitely many quasicyclic groups. We wish to specify a condition, analogous to the one above for torsion-free modules, that guarantees that $H^n(G,A)$ is a \v{C}ernikov group for all $n\geq 0$.
The precise result that we prove is Theorem B below. 

\begin{thmB}  Let $G$ be a group in $\mathcal{U}$ with spectrum $\pi$. Assume that $A$ is a $\mathbb ZG$-module whose underlying abelian group is a \v{C}ernikov $\omega$-group for a set of primes $\omega$. Suppose further that $A$ fails to have an infinite $\mathbb ZG$-submodule that is a quotient of a $\mathbb ZG$-module that is torsion-free of finite rank and $\pi$-spectral as an abelian group. Then, for all $n\geq 0$, $H^n(G,A)$ is a \v{C}ernikov $\omega$-group. 
\end{thmB}

\noindent Observe that the conclusion that $H^n(G,A)$ is a \v{C}ernikov group is the best possible outcome under the hypotheses of Theorem B: very easy examples reveal that the cohomology groups may not be finite. Also, there is no mention of homology in the theorem because $H_n(G,A)$ will always be a \v{C}ernikov $\omega$-group if $A$ is a \v{C}ernikov $\omega$-group and $G$ is solvable of finite rank. 

The proof of Theorem B, like that of its predecessor, draws extensively upon the techniques and findings of \cite{kroploren}. In addition, a key role is played by a theorem of J.C. Lennox and D. J. S. Robinson on the cohomology of nilpotent groups. The latter result enables us to reduce Theorem B to the case where $A$ lacks any infinite proper submodules and the subgroup $N$ from the definition of $\mathcal{U}$ acts trivially on $A$. It is shown in \cite{kroploren} that, provided $A$ is infinite, the triviality of the action of $N$ allows us to express the $E_2$-page of the LHS cohomology spectral sequence for $1\rightarrow N\rightarrow G\rightarrow Q\rightarrow 1$  as $E_2^{pq}={\rm Ext}^p_{\mathbb ZQ}(H_qN,A)$. As a result, we can obtain the conclusion of Theorem B by establishing that ${\rm Ext}^p_{\mathbb ZQ}(H_qN,A)$ is \v{C}ernikov. To accomplish this, we identify the following property of modules over integral group rings of finitely generated abelian groups. 

\begin{propA} Let $G$ be a finitely generated abelian group. Assume that $A$ is a $\mathbb ZG$-module that has finite rank {\it qua} abelian group. Suppose further that $B$ is a $\mathbb ZG$-module whose underlying additive group is a \v{C}ernikov $\omega$-group for a set of primes $\omega$. If $A$ fails to have an infinite $\mathbb ZG$-module quotient that is isomorphic to a quotient of $B$, then, for $n\geq 0$,  ${\rm Ext}_{\mathbb ZG}^n(A,B)$ is a \v{C}ernikov $\omega$-group.
\end{propA}

\noindent Similar to the use of [{\bf 3}, Proposition A] in the proof of Theorem A, the application of Proposition A to the $E_2$-page of the spectral sequence described above requires that $H_nN$ lies in the class $\mathcal{C}(Q,\pi)$. 

The plan of the paper is as follows. Proposition A is established  in \S 2; the proof relies on a generalization of [{\bf 3}, Proposition A], as well as an elementary property of Pontryagin duals of modules (Lemma 2.8). In \S 3 we discuss the class $\mathcal{C}(G,\pi)$ for a group $G$ and set of primes $\pi$, proving the pivotal Corollary 3.8. 
 The last section of the article is devoted to the proofs of Theorems A and B, with the former established right at the outset.  The proof of Theorem B demands some further preparation: most importantly, we must show that the condition on the $\mathbb ZG$-submodules of $A$ in the theorem also holds for $\mathbb ZG_0$-submodules, where $G_0$ is any subgroup of finite index (Lemma 4.2). After proving Theorem B and enunciating a corollary, we conclude the paper by presenting two examples that demonstrate the necessity of the hypotheses of the theorem.   

 Below we describe the notation and terminology that we will employ throughout the paper. In addition, we state a proposition summarizing the fundamental structural properties of solvable groups of finite rank. For the proofs of these properties, we refer the reader to \cite{robinson-lennox2}.
\vspace{10pt}

\begin{notation}{\rm

\vspace{5pt}

A {\it section} of a group is a quotient of one of its subgroups. 
\vspace{5pt}

Let $R$ be a ring and $A$ an $R$-module. An {\it $R$-module section} of $A$ is an $R$-module quotient of one of its $R$-submodules.
\vspace{5pt}

Let $R$ be a ring. An $R$-module $A$ is said to be {\it bounded} if there exists $r\in R$ such that $rA=0$. 
\vspace{5pt}

Suppose that $R$ is a principal ideal domain and $A$ an $R$-module. Let $F$ be the field of fractions of $R$. The {\it torsion-free rank} of $A$, denoted $r_0(A)$, is defined by $$r_0(A)={\rm dim}_F(A\otimes_RF).$$ Moreover, for each prime $p$ in $R$, the {\it $p$-rank}, $r_p(A)$, is defined by $$r_p(A)={\rm dim}_{R/pR}({\rm Hom}(R/pR,A)).$$ We say that $A$ has finite $R$-rank if $r_0(A)$ is finite and $r_p(A)$ is finite for every prime $p\in R$. 
\vspace{5pt}

If $p$ is a prime, then $\hat{\mathbb Z}_p$ is the ring of $p$-adic integers and $\mathbb Z_{p^\infty}$ the quasicyclic $p$-group.
\vspace{5pt}

Let $G$ be a group. If $g, x\in G$, then $x^g=g^{-1}xg$. Moreover, for $H<G$ and $g\in G$, $H^g=g^{-1}Hg$. 
\vspace{5pt}

The join of all the nilpotent normal subgroups of a group $G$, known as the {\it Fitting subgroup}, is denoted ${\rm Fitt}(G)$.
\vspace{5pt}

We designate the join of all the torsion normal subgroups of a group $G$ by $\tau(G)$.
\vspace{5pt}

A solvable group is said to have {\it finite rank} if its abelian sections all have finite $\mathbb Z$-rank. 
\vspace{5pt}

If $G$ is a solvable group such that $G/\tau(G)$ has finite rank, then $h(G)$ denotes the {\it Hirsch rank} of $G$, namely, the number of infinite cyclic factors in any series of finite length whose factors are all either infinite cyclic or torsion. By the Schreier refinement theorem, this number is independent of the particular series selected.
\vspace{5pt}

A solvable group is {\it minimax} if it has a series of finite length in which each factor is either cyclic or quasicyclic. If $\pi$ is a set of primes, then a {\it $\pi$-minimax group} is a solvable minimax group that is $\pi$-spectral. 
\vspace{5pt}

If $G$ is a \v{C}ernikov group, then the number of quasicyclic summands in the divisible subgroup of finite index is called the {\it \v{C}ernikov rank} of $G$, written $c(G)$. 
\vspace{5pt}

Let $G$ be a group and $A$ a $\mathbb ZG$-module. If $A\otimes_{\mathbb Z}\mathbb Q$ is a simple $\mathbb QG$-module, then we will refer to $A$ as {\it rationally irreducible}. In other words, $A$ is rationally irreducible if and only if the additive group of $A$ is not torsion and, for every submodule $B$ of $A$, either $B$ or $A/B$ is torsion as an abelian group.
\vspace{5pt}

Assume that $G$ is a group. Let $A$ be a $\mathbb ZG$-module and $B$ a subgroup of the underlying additive group of $A$. We define 

$$C_G(A)=\{g\in G: ga=a \  \ \forall\,  a\in A\}; \ \ \  N_G(B)=\{g\in G: gb\in B\  \ \forall\,  b\in B\}.$$}
\end{notation}
\vspace{10pt}

\begin{proposition} Let $G$ be a solvable group of finite rank. 
\vspace{5pt}

\noindent (i) $G$ is virtually torsion-free if and only if $\tau(G)$ is finite.
\vspace{5pt}

\noindent (ii) If $\tau(G)$ is a \v{C}ernikov group, then ${\rm Fitt}(G)$ is nilpotent and $G/{\rm Fitt}(G)$ virtually abelian. 
\vspace{5pt}

\noindent (iii) If $\tau(G)$ is finite, then $G/{\rm Fitt}(G)$ is finitely generated. 
\vspace{5pt}

\noindent (iv) If $G$ is finitely generated, then $G$ is minimax.  \nolinebreak \hfill\(\square\)

\end{proposition}

\section{Proposition A}

Proposition A is based on the following more general formulation of [{\bf 3}, Proposition A]. The original result was stated for the case $R=\mathbb Z$, but the proof applies to any principal ideal domain satisfying the finiteness condition given. 

\begin{proposition}{\rm (Kropholler, Lorensen, and Robinson)} Let $G$ be an abelian group and $R$ a principal ideal domain such that $R/Ra$ is finite for every  nonzero element $a$ of $R$. Assume that $A$ and $B$ are $RG$-modules that are $R$-torsion-free and have finite $R$-rank. Suppose further that $A$ fails to contain a nontrivial $RG$-submodule that is isomorphic to a submodule of $B$. Then  ${\rm Ext}_{RG}^n(A,B)$ and ${\rm Tor}_n^{RG}(A,B)$ are bounded $R$-modules
for all $n\geq 0$.  \nolinebreak \hfill\(\square\)
\end{proposition}

In proving Proposition 2.1, one requires the observation below from \cite{kroploren} concerning near splittings of modules. We will invoke this property in proving Proposition A, too.  

\begin{proposition} Let $R$ be a ring. Suppose that $A$ and $B$ are $R$-modules such that, for every positive integer $m$, both the kernel and the cokernel of the map $b\mapsto mb$ from $B$ to $B$ are finite.  Let $0\rightarrow B\rightarrow E\rightarrow A\rightarrow 0$  be an $R$-module extension and $\xi$ the element of ${\rm Ext}_R^1(A,B)$ that corresponds to this extension. Then $\xi$ has finite order if and only if $E$ has a submodule $X$ such that $B\cap X$ and $E/B+X$ are finite. \nolinebreak \hfill\(\square\)
\end{proposition}

Two other results from \cite{kroploren} that we will employ in this section, and in \S 4, are the following proposition and its corollary.

\begin{proposition} Let $G$ be a group and $R$ a commutative ring. Suppose that $A$ and $B $ are $RG$-modules and regard ${\rm Ext}^n_R(A,B)$ and ${\rm Tor}_n^R(A,B)$ as $\mathbb ZG$-modules via the diagonal action for $n\geq 0$. 
Then the following two statements hold. 
\vspace{5pt}

\noindent (i) There is a cohomology spectral sequence whose $E_2$-page is given by 
$$E_2^{pq}=H^p(G,{\rm Ext}^q_R(A,B)),$$
\noindent and that converges to ${\rm Ext}^n_{RG}(A,B)$.
\vspace{5pt}

\noindent (ii) There is a homology spectral sequence whose $E^2$-page is given by 
$$E^2_{pq}=H_p(G,{\rm Tor}_q^R(A,B)),$$
\noindent and that converges to ${\rm Tor}_n^{RG}(A,B)$. \hfill\(\square\)
\end{proposition}

\begin{corollary} Let $G$ be a group and $R$ a commutative ring. Suppose that $A$ and $B $ are $RG$-modules and view both ${\rm Hom}_R(A,B)$ and $A\otimes_R B$ as $\mathbb ZG$-modules under the diagonal actions. 

\noindent (i) If either $A$ is projective or $B$ is injective as an $R$-module, then, for $n\geq 0$,

$${\rm Ext}^n_{RG}(A,B)\cong H^n(G,{\rm Hom}_R(A,B)).$$

\noindent (ii) If $B$ is flat as an $R$-module, then, for $n\geq 0$,

$${\rm Tor}_n^{RG}(A,B)\cong H_n(G, A\otimes_R B).$$
\hfill\(\square\)
\end{corollary}

The argument for Proposition A  also makes use of the properties of modules described in the following two lemmas. The proofs of these are entirely elementary and therefore left to the reader.

\begin{lemma} Let $G$ be a group. Suppose that $D$ is a $\mathbb ZG$-module whose underlying abelian group is a divisible \v{C}ernikov group. Then any $\mathbb ZG$-module extension of $D$ by a finite module can be expressed as a $\mathbb ZG$-module extension of a finite module by a quotient of $D$.  \nolinebreak \hfill\(\square\)
\end{lemma}

\begin{lemma} Let $R$ be a principal ideal domain and $G$ a group. Assume that $A$ and $B$ are $RG$-modules such that $A$ is finitely generated as an $RG$-module. 
\vspace{5pt}

\noindent (i) If $B$ is finitely generated as an $R$-module, then ${\rm Hom}_{RG}(A,B)$ is finitely generated as an $R$-module. 
\vspace{5pt}

\noindent (ii) If $B$ has finite rank as an $R$-module, then ${\rm Hom}_{RG}(A,B)$ has finite rank as an $R$-module.
\nolinebreak \hfill\(\square\)

\end{lemma}

Before establishing Proposition A, we discuss a key tool employed in the proof, namely, the Pontryagin dual of a module, defined for the convenience of the reader below.

\begin{definition}{\rm Assume that $p$ is a prime. If $A$ is an abelian $p$-group with finite rank, then the {\it Pontryagin dual} $A^\ast$ of $A$ is the group ${\rm Hom}_{\mathbb Z}(A,\mathbb Z_{p^\infty})$. The dual $A^\ast$ is a finitely generated $\hat{\mathbb Z}_p$-module. If $A$ happens to be endowed with a $\mathbb ZG$-module structure for a group $G$, then we equip $A^\ast$ with a $G$-action by letting $(g \phi)(x)=\phi(g^{-1}x)$ for every $g\in G$, $\phi\in A^\ast$, and  $x\in A$. In this way, we may view $A^\ast$ as a $\hat{\mathbb Z}_pG$-module. 

For any finitely generated $\hat{\mathbb Z}_p$-module $M$, the {\it Pontryagin dual} $M^\ast$ of $M$ is the group ${\rm Hom}_{\mathbb Z}(M,\mathbb Z_{p^\infty})$.  In this case, the dual $M^\ast$ is an abelian $p$-group with finite rank. If $M$ is a $\mathbb ZG$-module for a group $G$, then we regard $M^\ast$ as a $\mathbb ZG$-module using the same $G$-action as above.
}

\end{definition}

Some elementary properties of Pontryagin duals are described in the following two lemmas. For proofs of these, we refer the reader to \cite{wilson}. 

\begin{lemma} Let $p$ be a prime. 
\vspace{5pt}

\noindent (i) If $A$ and $B$ are abelian $p$-groups of finite rank, then 

$${\rm Hom}_{\mathbb Z}(A, B)\cong {\rm Hom}_{\hat{\mathbb Z}_p}(B^\ast, A^\ast).$$
\vspace{5pt}

\noindent (ii) If $M$ and $N$ are finitely generated $\hat{\mathbb Z}_p$-modules,  then 

$${\rm Hom}_{\hat{\mathbb Z}_p}(M, N)\cong {\rm Hom}_{\mathbb Z}(N^\ast, M^\ast).$$ \nolinebreak \hfill\(\square\)

\end{lemma}

\begin{lemma} Let $p$ be a prime and $G$ a group. 
\vspace{5pt}

\noindent (i) If $A$ is a $\mathbb ZG$-module whose additive group is a $p$-group of finite rank, then $A$ is isomorphic to its double dual $A^{\ast \ast}$. 
\vspace{5pt}

\noindent (ii) If $M$ is a $\hat{\mathbb Z}_pG$-module that is finitely generated as a $\hat{\mathbb Z}_p$-module, then $M$ is isomorphic to its double dual $M^{\ast \ast}$. \nolinebreak \hfill\(\square\)
\end{lemma}

The crucial property of Pontryagin duals for our purposes is described in Lemma 2.9 below, which appeared in a preliminary version of \cite{kroploren}. This fact is undoubtedly well known; however, as we are unaware of any reference to it in the published literature, we include a proof. 

\begin{lemma} Assume that $p$ is a prime and $G$ a group. Let $A$ and $B$ be $\mathbb ZG$-modules whose additive groups are $p$-groups of finite rank. Suppose further that, as an abelian group, $B$ is divisible. Then

$${\rm Ext}_{\mathbb ZG}^n(A,B)\cong {\rm Ext}_{\hat{\mathbb Z}_pG}^n(B^\ast,A^\ast)$$
for $n\geq 0$.
\end{lemma}

\begin{proof} We have the following chain of isomorphisms.

$${\rm Ext}_{\mathbb ZG}^n(A,B)\cong H^n(G,{\rm Hom}_{\mathbb Z}(A,B))\cong H^n(G,{\rm Hom}_{\hat{\mathbb Z}_p}(B^\ast,A^\ast))\cong {\rm Ext}_{\hat{\mathbb Z}_pG}^n(B^\ast,A^\ast).$$
The first and last isomorphisms in this chain are consequences of Corollary 2.4(i). The first arises from the divisibility of the underlying abelian group of $B$ and the last from the resulting projectivity of $B^\ast$ as a $\hat{\mathbb Z}_p$-module. 
\end{proof}

We are now prepared to prove Proposition A. 

\begin{proof}[{\bf Proof of Proposition A}] Without incurring any significant loss of generality, we can assume that $\omega$ contains just a single prime $p$. We begin the proof with the observation that, if $B$ is finite, then ${\rm Ext}_{\mathbb ZG}^n(A,B)$ must be a finite $p$-group for all $n\geq 0$.
This follows immediately from Proposition 2.3(i) and the fact that ${\rm Ext}_{\mathbb Z}^n(A,B)$ is a finite $p$-group for all $n\geq 0$. Combining this observation with Lemma 2.5 permits a further reduction to the case where $B$ is divisible.

In the next two paragraphs, we consider the case where either $A$ is rationally irreducible or the additive group of $A$ is torsion. Under either of these assumptions, there is a finitely generated submodule $A_0$ of $A$ such that the additive group of $A/A_0$ is torsion. Since $\mathbb ZG$ is Noetherian, $A_0$ is of type ${\rm FP}_\infty$. Invoking Lemma 2.6(ii), we deduce that ${\rm Ext}^n_{\mathbb ZG}(A_0,B)$ is a \v{C}ernikov $p$-group for $n\geq 0$.  Now let $P$ be the $p$-torsion subgroup of $A/A_0$. 
 Our intention is to show that  ${\rm Ext}_{\mathbb ZG}^n(P,B)$ is a finite $p$-group for $n\geq 0$, which will yield the conclusion of the proposition.

To establish that ${\rm Ext}_{\mathbb ZG}^n(P,B)$ is a finite $p$-group, we use the fact that ${\rm Ext}_{\mathbb ZG}^n(P, B)$ and ${\rm Ext}_{\hat{\mathbb Z}_pG}^n(B^\ast,P^\ast)$ are isomorphic for all $n\geq 0$. Since $\hat{\mathbb Z}_pG$ is Noetherian and $B^\ast$ is a finitely generated $\hat{\mathbb Z}_pG$-module,  $B^\ast$ has type ${\rm FP}_{\infty}$ as a $\hat{\mathbb Z}_pG$-module. Thus, by Lemma 2.6(i),  ${\rm Ext}_{\hat{\mathbb Z}_pG}^n(B^\ast,P^\ast)$ is finitely generated as a $\hat{\mathbb Z}_p$-module. Suppose that ${\rm Ext}_{\hat{\mathbb Z}_pG}^n(B^\ast,P^\ast)$ is not a finite $p$-group for some $n\geq 0$. Then ${\rm Ext}_{\hat{\mathbb Z}_pG}^n(B^\ast,P^\ast)$ is not bounded as a $\hat{\mathbb Z}_p$-module. According to Proposition 2.1, this means that $P^\ast$ and $B^\ast$ have isomorphic nontrivial $\hat{\mathbb Z}_pG$-submodules. Dualizing a second time, we deduce that $P$ and $B$ have isomorphic infinite $\mathbb ZG$-module quotients, thereby contradicting our hypothesis. Therefore, for each $n\geq 0$, ${\rm Ext}_{\hat{\mathbb Z}_pG}^n(B^\ast,P^\ast)$, and hence ${\rm Ext}_{\mathbb ZG}^n(P, B)$, must be a finite $p$-group. This concludes the argument for the case where $A$ is either rationally irreducible as a module or torsion {\it qua} abelian group.

We handle the general case by inducting on $h(A)$, the case $h(A)=0$ having already been established above. Suppose $h(A)\geq 1$. Let $U$ be a submodule of $A$ with $h(U)$ as large as possible while still remaining less than $h(A)$. It follows from the case proved in the first paragraph that ${\rm Ext}^n_{\mathbb ZG}(A/U, B)$ is a \v{C}ernikov $p$-group for $n\geq 0$. Consequently, the conclusion of the theorem will follow from the inductive hypothesis if we succeed in showing that $U$ does not have a quotient that is isomorphic to a nontrivial quotient of $B$. To accomplish this, suppose that $U$ possesses such a quotient and take $V$ to be a submodule of $U$  such that $U/V\cong W$, where $W$ is a nontrivial quotient of $B$. By the case of the proposition established above, ${\rm Ext}^1_{\mathbb ZG}(A/U,U/V)$ is torsion. Appealing to Proposition~2.2, we deduce that $A$ has a quotient that is isomorphic to a finite extension of a nontrivial quotient of $W$. Lemma 2.5 implies, then, that there is a nontrivial quotient of $A$ that is isomorphic to a quotient of $W$. But the presence of such a quotient contradicts our hypothesis. Therefore, $U$ cannot possess a nontrivial quotient that is isomorphic to a quotient of $B$. This completes the proof of Proposition A.

\end{proof}

\section{The Class $\mathfrak{C}(G, \pi)$}

To prove Theorems A and B, we need to understand the following class of modules. This class is a slight extension of the class with the same designation discussed in \cite{kroploren}. 

\begin{definition} {\rm Assume that $G$ is a group and $\pi$ a set of primes. Let $\mathfrak{C}(G, \pi)$ be the smallest class of $\mathbb ZG$-modules with the following two properties.
\vspace{5pt}

\noindent (i) The class $\mathfrak{C}(G, \pi)$ contains every finite $\mathbb ZG$-module as well as every $\mathbb ZG$-module whose underlying abelian group is torsion-free of finite rank and $\pi$-spectral. 
\vspace{5pt}

\noindent (ii) The class $\mathfrak{C}(G, \pi)$ is closed under forming $\mathbb ZG$-module quotients and extensions.}
\end{definition}

The class $\mathfrak{C}(G, \pi)$ has the following two additional closure properties, which can be established by arguments in the same vein as those for [{\bf 3}, Lemmas 2.4, 2.5]. The basic strategy is to induct on the number of closure operations from (ii) that are necessary to construct the module from ones that are either finite or torsion-free of finite rank and $\pi$-spectral. We leave the small task of supplying complete proofs to the reader. 

\begin{lemma} For any group $G$ and set of primes $\pi$, the class $\mathfrak{C}(G, \pi)$ is closed under forming $\mathbb ZG$-module sections. \nolinebreak \hfill\(\square\)
\end{lemma} 

\begin{lemma} Assume that $G$ is a group and $\pi$ a set of primes. 
Let $A$ and $B$ be $\mathbb ZG$-modules  and, for each $n\geq 0$, regard ${\rm Tor}^{\mathbb Z}_n(A,B)$ as a $\mathbb ZG$-module under the diagonal action. If $A$ and $B$ belong to $\mathfrak{C}(G, \pi)$, then ${\rm Tor}^{\mathbb Z}_n(A,B)$ lies in $\mathfrak{C}(G, \pi)$ for $n\geq 0$. \nolinebreak \hfill\(\square\)
\end{lemma}

Of central importance for our purposes are the following two homological properties of the class $\mathfrak{C}(G, \pi)$. Both of these can be deduced in a manner similar to [{\bf 3}, Lemma 2.6], although for Lemma 3.4 one must invoke Proposition A of the present paper instead of [{\bf 3}, Proposition A].

\begin{lemma} Let $\pi$ be a set of primes and $G$ an abelian group. Assume that $B$ is a $\mathbb ZG$-module whose additive group is torsion-free with finite rank. Suppose further that there are no nontrivial $\mathbb ZG$-submodules of $B$ that are $\pi$-spectral as abelian groups. If $A$ is a $\mathbb ZG$-module in $\mathfrak{C}(G, \pi)$ and $\bar{B}$ a $\mathbb ZG$-module quotient of $B$, then ${\rm Ext}_{\mathbb ZG}^n(A,\bar{B})$ and ${\rm Tor}^{\mathbb ZG}_n(A,\bar{B})$ are bounded $\mathbb Z$-modules for $n\geq 0$. \nolinebreak \hfill\(\square\)
\end{lemma}

\begin{lemma} Let $\pi$ be a set of primes and $G$ a finitely generated abelian group. Assume that $B$ is a $\mathbb ZG$-module whose additive group is a \v{C}ernikov $\omega$-group for some set of primes $\omega$. Suppose further that there is no infinite $\mathbb ZG$-module quotent of $B$ that is also a quotient of a $\mathbb ZG$-module that is torsion-free of finite rank and $\pi$-spectral as an abelian group. If $A$ is a $\mathbb ZG$-module in $\mathfrak{C}(G, \pi)$, then, for each $n\geq 0$, ${\rm Ext}_{\mathbb ZG}^n(A, B)$ is a \v{C}ernikov $\omega$-group. \nolinebreak \hfill\(\square\)
\end{lemma}

Our main goal in this section is to venture beyond what was shown in \cite{kroploren} about the contents of $\mathfrak{C}(G, \pi)$ by showing, first, that the class is closed under the integral homology functor. This will then permit us to identify a large class of nilpotent normal subgroups of solvable groups of finite rank  whose integral homology lies within the eponymous class with respect to the ambient group. In order to establish the former property, we will avail ourselves of the results of Breen in \cite{breen} concerning the integral homology of abelian groups. We begin by defining the $m$-fold torsion product of abelian groups, following Breen's notation and terminology.

\begin{definition}{\rm  Assume that $A_1, \dots, A_m$ are abelian groups. For each $i=1, \dots, m$, let ${\bf F}^i$ be a $\mathbb Z$-flat resolution of $A_i$. Then, for $n\geq 0$, we define 

$${\rm Tor}^{\mathbb Z}_n(A_1, \dots, A_m)=H_n({\bf F}^1\otimes_{\mathbb Z} \dots \otimes_{\mathbb Z}{\bf F}^m).$$

\noindent If $m=2$, this group is isomorphic to ${\rm Tor}^{\mathbb Z}_n(A_1, A_2)$ as it is usually defined.} 

\end{definition}

Lemma 3.2 can be generalized to these higher order torsion products.

\begin{lemma} Assume that $G$ is a group and $\pi$ a set of primes. 
Let  $A_1, \dots, A_m$ be $\mathbb ZG$-modules  and, for each $n\geq 0$, regard ${\rm Tor}^{\mathbb Z}_n(A_1, \dots, A_m)$ as a $\mathbb ZG$-module under the diagonal action. If $A_1, \dots , A_m$ belong to $\mathfrak{C}(G, \pi)$, then ${\rm Tor}^{\mathbb Z}_n(A_1, \dots, A_m)$ lies in $\mathfrak{C}(G, \pi)$ for $n\geq 0$. 
\end{lemma}

\begin{proof}  We proceed by induction on $m$, the case $m=2$ having been disposed of in Lemma 3.2. Suppose $m>2$. By the K\"unneth formula ([{\bf 2}, Theorem V.2.1]), there is a short exact sequence

$$0\longrightarrow {\rm Tor}^\mathbb Z_n(A_1,\cdots A_{m-1})\otimes_{\mathbb Z} A_m\longrightarrow {\rm Tor}^{\mathbb Z}_n(A_1, \cdots, A_m)\longrightarrow {\rm Tor}_1^{\mathbb Z}({\rm Tor}^\mathbb Z_{n-1}(A_1,\cdots A_{m-1}), A_m)\longrightarrow 0$$

\noindent of $\mathbb ZG$-modules for $n\geq 0$. In view of the inductive hypothesis, ${\rm Tor}^\mathbb Z_i(A_1,\cdots A_{m-1})$ belongs to $\mathfrak{C}(G,\pi)$ for all $i\geq 0$.  Invoking Lemma 3.2 and the closure of $\mathfrak{C}(G,\pi)$ under extensions, we deduce that ${\rm Tor}^{\mathbb Z}_n(A_1, \dots, A_m)$ must reside in $\mathfrak{C}(G,\pi)$ for $n\geq 0$.
\end{proof}

Equipped with the preceding lemma, we can apply two of the results in \cite{breen} to obtain our most important closure property for $\mathfrak{C}(G, \pi)$.

\begin{lemma} Assume that $G$ is a group and $\pi$ a set of primes. If $A$ is a $\mathbb ZG$-module that belongs to $\mathfrak{C}(G,\pi)$, then $H_nA$ lies in $\mathfrak{C}(G,\pi)$ for $n\geq 0$. 
\end{lemma}

\begin{proof} According to [{\bf 1}, (1.10)], there is a natural spectral sequence $\{E^r_{pq}\}$ converging to $H_nA$ such that $E^2_{pq}=L_p\Lambda^qA$ for $p, \ q\geq 0$. Here $L_p\Lambda^qA$ represents the $p$th derived functor of the $q$th exterior power of $A$. Moreover, by [{\bf 1}, Theorem 4.7], there is a $\mathbb ZG$-module monomorphism from $L_p\Lambda^qA$ to ${\rm Tor}^{\mathbb Z}_p(\underbrace{A, \dots, A)}_q$ for $q>0$. It follows, then, from Lemmas 3.5 and 3.1 that $E^2_{pq}$ lies in $\mathfrak{C}(G,\pi)$ for $q>0$. In addition, $E^2_{p0}$, being trivial, belongs to $\mathfrak{C}(G,\pi)$. Since the differentials in Breen's spectral sequence are $\mathbb ZG$-module homomorphisms, $E^\infty_{pq}$ can be regarded as a $\mathbb ZG$-module section of $E^2_{pq}$. Hence $E^\infty_{pq}$ must be a member of $\mathfrak{C}(G,\pi)$ for $p,\ q\geq 0$.  Furthermore, $H_nA$ has a $\mathbb ZG$-module series whose factors are the modules $E^\infty_{pq}$ for $p+q=n$. Therefore, $H_nA$ must belong to $\mathfrak{C}(G,\pi)$ for $n\geq 0$.
\end{proof}

Now we prove the desired result concerning the integral homology of a nilpotent normal subgroup.

\begin{proposition} Assume that $G$ is a group with a nilpotent normal subgroup $N$ of finite rank. Let $\pi={\rm spec}(N)$ and $Q=G/N$. Suppose further that $N$ has a $G$-invariant central series in which each infinite factor can be expressed as a quotient of a $\mathbb ZQ$-module whose underlying abelian group is torsion-free of finite rank and $\pi$-spectral. Then $H_nN$ belongs to $\mathfrak{C}(Q, \pi)$ for $n\geq 0$.

\end{proposition}

\begin{proof} Let $1=N_0<N_1<\dots <N_r=N$ be a central series with the properties described. We prove the proposition by induction on $r$, the case $r=0$ being trivial. Assume $r>0$. By the inductive hypothesis, $H_n(N/N_1)$ is a member of $\mathfrak{C}(Q, \pi)$ for $n\geq 0$.  Moreover, according to Lemma 3.6, $H_nN_1$ lies in  $\mathfrak{C}(Q, \pi)$ for $n\geq 0$.  We now consider the LHS homology spectral sequence associated to the central extension $1\rightarrow N_1\rightarrow N\rightarrow N/N_1\rightarrow 1$. In this spectral sequence, $E_{pq}^2=H_p(N/N_1, H_qN_1)$, and so there is an exact sequence

$$0\longrightarrow H_p(N/N_1)\otimes_{\mathbb Z} H_qN_1\longrightarrow E_{pq}^2\longrightarrow {\rm Tor}_1^{\mathbb Z}(H_{p-1}(N/N_1), H_qN_1)\longrightarrow 0.$$

\noindent An appeal to Lemma 3.2 allows us to deduce that $E_{pq}^2$ lies in $\mathfrak{C}(Q, \pi)$.  Since the differentials in the spectral sequence are $\mathbb ZQ$-module homomorphisms, it follows from Lemma 3.1 that $E_{pq}^\infty$ belongs to $\mathfrak{C}(Q, \pi)$.  As a result, $H_nN$ is a member of $\mathfrak{C}(Q, \pi)$ for $n\geq 0$.

\end{proof}

As an immediate consequence of Proposition 3.7, we have 

\begin{corollary} Assume that $G$ is a group in the class $\mathcal{U}$ and $N$ and $Q$ are as described in the definition of $\mathcal{U}$. Then $H_nN$ belongs to $\mathfrak{C}(Q, \pi)$ for $n\geq 0$.  \nolinebreak \hfill\(\square\)
\end{corollary}

We conclude this section by highlighting an important type of group that belongs to $\mathcal{U}$.

\begin{proposition}  Every finitely generated solvable minimax group lies in $\mathcal{U}$.
\end{proposition}

\begin{proof} Let $G$ be a finitely generated solvable minimax group with spectrum $\pi$. Write $N={\rm Fitt}(G)$ and $Q=G/N$. Let $i$ be a positive integer and consider the canonical $\mathbb ZQ$-module epimorphism $\theta_i: \underbrace{N_{\rm ab}\otimes_{\mathbb Z} \dots \otimes_{\mathbb Z} N_{\rm ab}}_{i}\rightarrow \gamma_iN/\gamma_{i+1}N$. Since $G$ is finitely generated as a group and $Q$ polycyclic, $N_{\rm ab}$ is finitely generated as a $\mathbb ZQ$-module. As a result,  $N_{\rm ab}$ is a Noetherian $\mathbb ZQ$-module and hence virtually torsion-free as an abelian group. Therefore, each factor in the lower central series of $N$ is a homomorphic image of a $\mathbb ZQ$-module whose underlying additive group is virtually torsion-free of finite rank and $\pi$-spectral. This, then, implies that $G$ belongs to $\mathcal{U}$ .

\end{proof}

\section{Theorems A and B}

This section is devoted to the proofs of the main results of the paper, Theorems A and B.  We begin with the former, as it presents the least difficulty. 

\begin{proof}[{\bf Proof of Theorem A}]

As pointed out at the beginning of the proof of [{\bf 3}, Theorem A], the conclusion will follow if we can show that the cohomology and homology groups are bounded. Let $N$ be a nilpotent normal subgroup of $G$ with the two properties described in the definition of $\mathcal{U}$. Put $Q=G/N$. 
First we treat the case where $A$ is rationally irreducible. Under this assumption, $G/C_G(A)$ is virtually abelian due to [{\bf 5}, 3.1.6]. Furthermore, as argued in the proof of [{\bf 5}, 5.2.3], the structure of the multiplicative group of an algebraic number field allows us to conclude that $G/C_G(A)$ is finitely generated. Hence, by replacing $N$ with $N\cap C_G(A)$, we can assume $N<C_G(A)$. Let $G_0$ be a subgroup of $G$ with finite index such that $Q_0=G_0/N$ is abelian. Since $H_nN$ lies in $\mathfrak{C}(Q_0,\pi)$ by Corollary 3.8, we can argue exactly as in the second and third paragraphs of the proof of [{\bf 3}, Theorem A], except that now we invoke Lemma 3.3 rather than [{\bf 3}, Lemma 2.6]. This reasoning allows us to deduce that  $H^n(G_0,A)$ and $H_n(G_0,A)$ are bounded for $n\geq 0$. Thus the same holds for $H^n(G,A)$ and $H_n(G,A)$. 

We handle the general case by inducting on $h(A)$, the case $h(A)=0$ being trivial. Assume $h(A)\geq 1$. Choose $B$ to be a submodule of $A$ with $h(B)$ as large as possible subject to the conditions that $h(B)<h(A)$ and the additive group of $A/B$ is torsion-free. By the inductive hypothesis, $H^n(G,B)$ and $H_n(G,B)$ are finite for $n\geq 0$. Since $A/B$ is rationally irreducible, the conclusion will then follow from the case proved above if we can show that $A/B$ fails to contain a nonzero $\mathbb ZG$-submodule that is torsion-free and $\pi$-spectral as an abelian group. Suppose that $A/B$ contains such a submodule, and let $C$ be a submodule of $A$ properly containing $B$ such that $C/B$ is torsion-free and $\pi$-spectral {\it qua} abelian group. Because $G/C_G(C/B)$ is finitely generated and virtually abelian, we may assume $N<C_G(C/B)$. Now, as in the proof of [{\bf 3}, Theorem A], we put $\Gamma=C\rtimes G$. It is straightforward to deduce that $\Gamma/B$ belongs to the class $\mathcal{U}$.  As a result, $H^2(\Gamma/B,B)$ is finite. This allows us to derive a contradiction by adducing the argument employed at the end of the proof of [{\bf 3}, Theorem A]. Therefore, we are compelled to conclude that $A/B$ fails to possess a nonzero $\mathbb ZG$-submodule that is torsion-free and $\pi$-spectral as an abelian group. This completes the proof of the theorem. 
\end{proof}

Combining Theorem A with Proposition 3.9, we obtain

\begin{corollary} Let $G$ be a finitely generated solvable minimax group with spectrum $\pi$. Assume that $A$ is a $\mathbb ZG$-module whose underlying abelian group is torsion-free and has finite rank. Suppose further that $A$ does not have any nontrivial $\mathbb ZG$-submodules that are $\pi$-spectral as abelian groups. Then $H^n(G,A)$ and $H_n(G,A)$ are finite 
for all $n\geq 0$.  \nolinebreak \hfill\(\square\)

\end{corollary}

The necessity of the condition on the submodules of $A$ in Theorem A is immediately apparent if we take $G=C_{\infty}\times C_{\infty}$ and $A=\mathbb Z$, with $G$ acting trivially. In this situation, $H^0(G,A)=\mathbb Z$, $H^1(G,A)=\mathbb Z\oplus \mathbb Z$, and $H^2(G,A)=\mathbb Z$. Rather less obvious is that the hypothesis that $G$ belongs to $\mathcal{U}$ can also not be discarded.  To verify that, we present the example below, which offers a cohomological interpretation of  [{\bf 3}, Example 3.3].

\begin{example} {\rm Let $p$ and $q$ be distinct primes.  Take $Q$ to be an infinite cyclic group, and let $B$ be the $\mathbb ZQ$-module obtained by letting the generator of $Q$ act on $\mathbb Z_{p^\infty}$ by multiplication by $q$. Put $G=B\rtimes Q$.  Let $A$ be the underlying additive group of the ring $\mathbb Z[1/q]$ and endow $A$ with a $\mathbb ZG$-module structure by letting the generator of $Q$ act  by multiplication by $q$ and $B$ centralize $A$.  We will employ the LHS spectral sequence associated to the extension $1\rightarrow B\rightarrow G\rightarrow Q\rightarrow 1$ to compute $H^2(G,A)$. In this spectral sequence,  $E^{11}_2=0$, and the universal coefficient theorem yields $E^{02}_2\cong {\rm Ext}_{\mathbb Z}^1(B,A)^Q$.  It is straightforward to ascertain ${\rm Ext}^1_{\mathbb Z}(B,A)\cong \hat{\mathbb Z}_p$. Therefore, since the diagonal action of $Q$ on ${\rm Ext}^1_{\mathbb Z}(B,A)$ is trivial, we deduce $H^2(G,A)\cong \hat{\mathbb Z}_p$. Thus we have described an example where all the hypotheses of Theorem A are satisfied save the stipulation that $G$ is a member of \ $\mathcal{U}$, and yet the cohomology in dimension two is torsion-free and nontrivial.
}

\end{example}

To lay the groundwork for Theorem B, we prove three lemmas concerning modules that can be expressed as quotients of modules that are torsion-free of finite rank.

\begin{lemma} Suppose that $\pi$ is a set of primes. Let $G$ be a group and $G_0<G$ with $[G:G_0]<\infty$. Assume that $A$ is a $\mathbb ZG$-module that can be expressed as a $\mathbb ZG_0$-module quotient of a $\mathbb ZG_0$-module whose underlying abelian group is torson-free, has finite rank, and is $\pi$-spectral.  Then $A$ can also be realized as a $\mathbb ZG$-module quotient of a $\mathbb ZG$-module whose additive group possesses the same three properties. 
\end{lemma}

\begin{proof} Assume that  $M_0$ is a $\mathbb ZG_0$-module that is torsion-free of finite rank and $\pi$-spectral as an abelian group, and that there is a $\mathbb ZG_0$-module epimorphism $\phi_0: M_0\to A$. Set $\Gamma=A\rtimes G$, $\Gamma_0=A\rtimes G_0$, and $\Omega_0=M_0\rtimes G_0$. Moreover, let $\phi_0^\ast: \Omega_0\to \Gamma_0$ be the epimorphism induced by $\phi_0$. In addition, suppose that $\theta:F\to \Gamma$ is a group epimorphism such that $F$ is a free group, and put $F_0=\theta^{-1}(\Gamma_0)$ and $N=\theta^{-1}(A)$. Since $F_0$ is free, there is a homomorphism $\psi_0: F_0\to \Omega_0$ that makes the diagram

\begin{displaymath}\xymatrix{
&\Omega_0  \ar[r]^{\phi_0^\ast} &\Gamma_0\\
& F_0 \ar[u]_{\psi_0}  \ar[ur]_{\theta_0} &}
\end{displaymath}
commute, where $\theta_0:F_0\to \Gamma_0$ is the epimorphism induced by $\theta$. Setting $K={\rm Ker}\ \theta$ and $K_0={\rm Ker}\ \psi_0$, we have the chain of subgroups $N'< K_0<K<N<F_0<F$. Now let $L=\bigcap_{f\in F} K_0^f$. Then $L\lhd F$ and $N'<L<K_0$.  Because $[F:F_0]$ is finite, there are only finitely many distinct conjugates of $K_0$ in $F$. Furthermore, for each $f\in F$, $N/K_0^f$ is a torsion-free abelian group of finite rank that is $\pi$-spectral. As a result, $N/L$ must enjoy the same properties. In addition, $K$ acts trivially on $N/L$ by conjugation. This means that we can regard $N/L$ as a $\mathbb Z\Gamma$-module, and hence as a $\mathbb ZG$-module via the map $G\to \Gamma$. Denoting this $\mathbb ZG$-module by $M$, we then have a $\mathbb ZG$-module epimorphism $\phi:M\to A$ induced by $\theta$. The lemma has now been proven. 
\end{proof}

\begin{lemma} Suppose that $\pi$ is a set of primes. Let $G$ be a group and $G_0<G$ with $[G:G_0]<\infty$. Assume that $A$ is a $\mathbb ZG$-module and $B$ a $\mathbb ZG_0$-submodule of $A$. Suppose further that $B$ can be expressed as a quotient of a $\mathbb ZG_0$-module whose underlying abelian group is torson-free, has finite rank, and is $\pi$-spectral.  Then the $\mathbb ZG$-submodule of $A$ generated by $B$ can be realized as a $\mathbb ZG$-module quotient of a $\mathbb ZG$-module whose additive group enjoys the same three properties.
\end{lemma}

\begin{proof}  Since $G_0$ contains a $G$-invariant subgroup of finite index, it suffices to consider the case where $G_0\lhd G$. Let $g_1,\dots, g_r$ be a complete list of coset representatives of $G_0$ in $G$.  Then $\bar{B}=g_1B+\dots +g_r B$ is the $\mathbb ZG$-submodule of $A$ generated by $B$. According to our hypotheses, there is a $\mathbb ZG_0$-module $M$ and a $\mathbb ZG_0$-module epimorphism $\phi:M\to B$. Now, for each $i=1, \dots, r$, we construct a new $\mathbb ZG_0$-module $M_i$ with the same underlying abelian group as $M$. The new action, denoted $\circ_i$ for each $i$, is defined in terms of the old action as follows: $g_0\circ_i x=g_0^{g_i}x$ for all $x\in M_i$ and $g_0\in G_0$. Next set $\bar{M}=  M_1\oplus \dots \oplus M_r$, and define the map $\bar{\phi}:\bar{M}\to \bar{B}$ by
$$\bar{\phi}(x_1,\dots, x_r)=g_1\phi (x_1)+\dots + g_r\phi(x_r)$$
for all $(x_1, \dots, x_r)\in \bar{M}$. It is easily checked that $\bar{\phi}$ is an $\mathbb ZG_0$-module epimorphism. Therefore, the conclusion of the lemma follows by applying Lemma 4.1.

\end{proof}

\begin{lemma} Assume that $\pi$ is a set of primes.  Let $G$ be a group and $A$ a $\mathbb ZG$-module that is divisible as an abelian group. Suppose that $A$ has a finite submodule $B$, and that there is an epimorphism $\theta: U\to A/B$, where $U$ is a $\mathbb ZG$-module whose additive group is torsion-free of finite rank and $\pi$-spectral. Then the module $A$, too, can be realized as a quotient of a $\mathbb ZG$-module whose additive group is torsion-free of finite rank and $\pi$-spectral.

\end{lemma}

\begin{proof}  From the map $\theta: U\to A/B$, we can construct a $\mathbb ZG$-module extension $0\rightarrow B\rightarrow M\rightarrow U\rightarrow 0$ that fits into a commutative diagram of the form

\begin{displaymath} \begin{CD}
0 @>>> B@>>> M @>>> U @>>>0\\
&& @| @VVV @VV{\theta}V &&\\ 
0 @>>> B @>>> A@>>> A/B @>>>0.
\end{CD} \end{displaymath}
Since $M$ is virtually torsion-free, we can select a torsion-free submodule $M_0$ of $M$ with finite index such that $M_0\cap B=0$. The submodule $M_0$, then, is mapped surjectively to $A$, thereby witnessing the property desired. 
\end{proof}

In addition to the preceding three lemmas, the proof of Theorem B depends upon three homological properties. The first two are contained in Lemma 3.4 and Corollary 3.8, respectively; the third is articulated in the following special case of a theorem due to 
Lennox and Robinson (see \cite{robinson-lennox1} and [{\bf 5}, 10.3.6]).

\begin{theorem}{\rm (Lennox and Robinson)} Assume that $\omega$ is a set of primes. Let $N$ be a nilpotent group with finite rank and $A$ a $\mathbb ZN$-module whose underlying additive group is a \v{C}ernikov $\omega$-group. If $A^N$ is finite, then $H^n(N,A)$ is a finite $\omega$-group for $n\geq 0$. \nolinebreak \hfill\(\square\)
\end{theorem}

In the original statement of Lennox and Robinson's theorem, $N$ is not assumed to have finite rank, which leads to the weaker conclusion that $H^n(N,A)$ is merely bounded as a $\mathbb Z$-module. However, as argued in the first paragraph in the proof of [{\bf 3}, Theorem A], $N$ and $A$ having finite rank means that, if $H^n(N,A)$ is bounded, then it must be finite.

Armed with the results above, we proceed with the proof of our second theorem.

\begin{proof} [{\bf Proof of Theorem B}] 
 
Assume that $N$ is the normal subgroup of $G$ described in the definition of $\mathcal{U}$ and write $Q=G/N$. 
Without any real loss of generality, we may assume that $A$ is divisible {\it qua} abelian group. In this paragraph, we treat the case where, for every subgroup $H$ of $G$ of finite index,   
each proper $\mathbb ZH$-submodule of $A$ is finite. Consider first the situation where $A^N\neq A$. Under this assumption, Lennox and Robinson's result yields that $H^n(N,A)$ is a finite $\omega$-group for $n\geq 0$. It follows, then, from the LHS spectral sequence that $H^n(G,A)$ is a finite $\omega$-group for $n\geq 0$.  Next we suppose $N<C_G(A)$.  Let $G_0$ be a normal subgroup of $G$ with finite index such that $N<G_0$ and $Q_0=G_0/N$ is a finitely generated abelian group.  
We will employ the LHS spectral sequence for the group extension $1\to N\to G_0\to Q_0\to 1$ to investigate $H^n(G_0, A)$. The universal coefficient theorem and Corollary 2.4 give rise to the following chain of isomorphisms for $p, \ q\geq 0$.
\begin{displaymath} H^p(Q_0, H^q(N, A)) \cong H^p(Q_0 , {\rm Hom}_{\mathbb Z}(H_qN, A))\cong {\rm Ext}^p_{\mathbb ZQ_0}(H_qN, A).\end{displaymath}
According to Lemmas 4.4 and 4.5, $A$ fails to have a quotient that is also a quotient of a $\mathbb ZQ_0$-module that is torsion-free of finite rank and $\pi$-spectral as an abelian group. Moreover, Corollary 3.8 implies that $H_qN$ belongs to $\mathfrak{C}(Q_0, \pi)$ for $q\geq 0$. Therefore, by Lemma 3.4, ${\rm Ext}^p_{\mathbb ZQ_0}(H_qN, A)$ is a \v{C}ernikov $\omega$-group for $p,\ q\geq 0$. It follows, then, that $H^n(G_0, A)$ is a \v{C}ernikov $\omega$-group for $n\geq 0$.  Finally, appeals to the LHS spectral sequence for the extension $1\rightarrow G_0\rightarrow G\rightarrow G/G_0\rightarrow 1$ and [{\bf 6}, Lemma 4.3] permit the deduction that $H^n(G,A)$ is a \v{C}ernikov $\omega$-group for $n\geq 0$. 

Now we tackle the general case, proceeding by induction on $c(A)$. As the case $c(A)=0$ is trivial, we suppose $c(A)>0$. Select an additive subgroup $B$ of $A$ with $c(B)$ as large as possible subject to the conditions that $c(B)<c(A)$ and $N_G(B)$ has finite index in $G$. Let $G_1$ be a normal subgroup of $G$ contained in $N_G(B)$ such that $[G:G_1]$ is finite. Notice that our choice of $B$ guarantees that, for any subgroup $H$ of $G_1$ with finite index, every proper $\mathbb ZH$-submodule of $A/B$ is finite. Furthermore, in view of Lemma 4.4 and the inductive hypothesis, $H^n(G_1,B)$ is a \v{C}ernikov $\omega$-group for $n\geq 0$. Hence the conclusion of the theorem will follow  if we manage to establish that $H^n(G_1,A/B)$ is a \v{C}ernikov $\omega$-group for $n\geq 0$. Set $N_1=G_1\cap N$. If $N_1$ fails to act trivially on $A/B$, then we deduce that $H^n(G_1,A/B)$ is a finite $\omega$-group from Lennox and Robinson's theorem.  Suppose that $N_1$ acts trivially on $A/B$. 
Our plan is to show that, in this case, $A/B$ cannot be a $\mathbb ZG_1$-module quotient of a $\mathbb ZG_1$-module whose additive group is torsion-free of finite rank and $\pi$-spectral; then it will follow from the case proved above that $H^n(G_1,A/B)$ is a \v{C}ernikov $\omega$-group for $n\geq 0$, thus yielding the desired conclusion. 

Suppose that $A/B$ is a quotient of a $\mathbb ZG_1$-module whose additive group is torsion-free of finite rank and $\pi$-spectral.
 Let $\Gamma=A\rtimes G_1$. Then the group $\Gamma/B$ plainly belongs to $\mathcal{U}$.  Consequently, by the inductive hypothesis, $H^n(\Gamma/B,B)$ is torsion for $n\geq 0$. That this holds for $n=2$ implies the existence of a subgroup $X$ of $\Gamma$ such that $B\cap X$ is finite and $[\Gamma:BX]<\infty$. Writing $G_2=BX$, we have that $A\cap X$ and $B\cap X$ are $\mathbb ZG_2$-submodules of $A$. Moreover, $A\cap X/B\cap X$ and $A/B$
are isomorphic as $\mathbb ZG_2$-modules. Thus Lemma 4.5 yields that $A\cap X$ is a $\mathbb ZG_2$-module quotient of a $\mathbb ZG_2$-module whose underlying additive group is torsion-free of finite rank and $\pi$-spectral.  In view of Lemma 4.4, this contradicts our hypothesis concerning $A$. Therefore, $A/B$ cannot be a $\mathbb ZG_1$-module quotient of a $\mathbb ZG_1$-module whose additive group is torsion-free of finite rank and $\pi$-spectral. This concludes the proof of the theorem.

\end{proof}

Theorem B has the following corollary.

\begin{corollary}  Let $G$ be a finitely generated solvable minimax solvable group with spectrum $\pi$. Assume that $A$ is a $\mathbb ZG$-module whose underlying abelian group is a \v{C}ernikov $\omega$-group for a set of primes $\omega$. Suppose further that $A$ fails to have an infinite $\mathbb ZG$-submodule that is a quotient of a $\mathbb ZG$-module that is torsion-free of finite rank and $\pi$-spectral as an abelian group. Then, for all $n\geq 0$, $H^n(G,A)$ is a \v{C}ernikov $\omega$-group.  \nolinebreak \hfill\(\square\)
\end{corollary}

In a second corollary below, we identify a common situation in which the condition on the submodules of $A$ in Theorem B applies, namely, where the action of the group on the module is transcendental.

\begin{corollary} Let $G$ be a group in the class $\mathcal{U}$ with spectrum $\pi$. Assume that $A$ is a $\mathbb ZG$-module whose additive group is a \v{C}ernikov $\omega$-group for a set of primes $\omega$, and whose proper submodules are all finite. Suppose further that $\phi: G\to {\rm Aut}(A)$ is the homomorphism arising from the action of $G$ on $A$. If, for some $g\in G$, $\phi(g)$ fails to satisfy any polynomial over $\mathbb Z$, then, for $n\geq 0$, $H^n(G,A)$ must be a \v{C}ernikov $\omega$-group.
\end{corollary}

\begin{proof} If $A$ could be realized as a quotient of a $\mathbb ZG$-module whose additive group were torsion-free of finite rank, then $\phi(g)$ would plainly satisfy a polynomial over $\mathbb Z$ for every $g\in G$. Hence it must be impossible to express $A$ as such a quotient. Theorem A, therefore, yields that $H^n(G,A)$ is a \v{C}ernikov $\omega$-group for $n\geq 0$. 
\end{proof}

We conclude the paper with two examples illustrating that, in Theorem B, we cannot dispense with either the restriction on $G$ or the condition on the submodules of $A$.

\begin{example} {\rm Let $p$ and $q$ be distinct primes.  Take $Q$ to be an infinite cyclic group, and let $A$ be the $\mathbb ZQ$-module obtained by letting the generator of $Q$ act on $\mathbb Z_{p^\infty}$ by multiplication by $q$. Set $G=A\rtimes Q$. We may also regard $A$ as a $\mathbb ZG$-module via the epimorphism $G\to Q$. First we claim that $A$ fails to be a quotient of a torsion-free $\mathbb ZG$-module that has finite rank and is $\{p\}$-spectral. To show this, we let $B$ be the underlying additive group of the ring $\mathbb Z[1/q]$ and endow $B$ with a $\mathbb ZQ$-module structure by letting the generator of $Q$ act again by multiplication by $q$. By Proposition 2.3(i), 
$${\rm Ext}_{\mathbb ZQ}^1(A,B)\cong {\rm Ext}^1_{\mathbb Z}(A,B)^Q.$$
Moreover, a simple calculation reveals  ${\rm Ext}^1_{\mathbb Z}(A,B)\cong \hat{\mathbb Z}_p$. Since the diagonal action of $Q$ on ${\rm Ext}^1_{\mathbb Z}(A,B)$ is trivial, this means
${\rm Ext}_{\mathbb ZQ}^1(A,B)\cong \hat{\mathbb Z}_p$. However, according to Lemma 3.3, if $A$ were a quotient of a torsion-free $\{p\}$-minimax $\mathbb ZG$-module, then ${\rm Ext}_{\mathbb ZQ}^1(A,B)$ would have to be bounded. Hence $A$ cannot be a quotient of such a module. 

Now we apply the LHS spectral sequence $\{E^{ij}_r\}$ associated to the group extension $1\rightarrow A\rightarrow G\rightarrow Q\rightarrow 1$ to calculate $H^n(G,A)$ for $n=1,\ 2$.  Since $E^{10}_2=H^1(Q,A)=0$, we deduce $H^1(G,A)\cong {\rm Hom}_{\mathbb Z}(A,A)^Q$. With the diagonal action being trivial, this yields $H^1(G,A)\cong \hat{\mathbb Z}_p$. Next we observe 
$$E^{11}_2\cong H^1(Q, {\rm Hom}_{\mathbb Z}(A,A))\cong \hat{\mathbb Z}_p.$$ Furthermore, invoking the universal coefficient theorem, we obtain $$E^{02}_2\cong {\rm Hom}_{\mathbb Z}(A,A)^Q\cong \hat{\mathbb Z}_p.$$
Hence $H^2(G,A)\cong \hat{\mathbb Z}_p\oplus \hat{\mathbb Z}_p$. 
This example, then, fulfills all  the hypotheses of Theorem B except the condition that $G$ is a member of $\mathcal{U}$, and the cohomology is torsion-free and uncountable in dimensions one and two.

}

\end{example}

\begin{example}{\rm Let $p$ be a prime and $Q=\langle t\rangle$ an infinite cyclic group. Form a $\mathbb ZQ$-module $B_1$ by letting $t$ act on $\mathbb Z[1/p]$ by multiplication by $p$. Moreover, let $B_2$ be the $\mathbb ZQ$-module with the same underlying additive group but where $t$ acts instead by multiplication by $1/p$. Next set $G=(B_1\oplus B_2)\rtimes Q$. Then $G$ is a finitely generated torsion-free solvable minimax group. Take $A$ to be the trivial $\mathbb ZG$-module with underlying abelian group $\mathbb Z_{p^\infty}$. We will employ the LHS spectral sequence $\{E^{ij}_r\}$ associated to the extension  $1\rightarrow B_1\oplus B_2\rightarrow G\rightarrow Q\rightarrow 1$ to compute $H^2(G,A)$. For this purpose, we let $V={\rm Hom}_{\mathbb Z}(\mathbb Z[1/p], A)$, so that $V$ is a vector space over $\mathbb Q$ with dimension equal to the cardinality of the continuum. We thus have $H^1(B_1\oplus B_2,A)\cong V\oplus V$, where $Q$ acts on the first copy of $V$ via  multiplication by $p$ and on the second by multiplication by $1/p$. Under this action, every derivation from $Q$ to $V\oplus V$ is inner, rendering $E^{11}_2=0$. Also, the universal coefficient theorem and K\"unneth formula yield $H^2(B_1\oplus B_2,A)\cong V$, where the action of $Q$ on $V$ is trivial. This means $E^{02}_2\cong V$. Consequently, $H^2(G,A)\cong V$. This example, then, demonstrates the necessity of the condition in Theorem B  on the submodules of $A$. }

\end{example}

\begin{acknowledgement} {\rm The author is greatly indebted to Peter Kropholler for many stimulating and enlightening exchanges concerning solvable groups. In addition, the author would like to express his gratitude to the anonymous referee for a number of salutary suggestions.}
\end{acknowledgement}


\begin{thebibliography}{7}




\bibitem[{\bf 1}]{breen}{\sc L. Breen.} On the functorial homology of abelian groups. {\it J. Pure Appl. Algebra}~{\bf 142} (1999), 199-237. 

\bibitem[{\bf 2}]{hilton}{\sc P. J. Hilton} and {\sc U. Stammbach}. {\it A Course in Homological Algebra}, Second Edition (Springer, 1997). 

\bibitem[{\bf 3}]{kroploren}{\sc P. H. Kropholler} and {\sc K. Lorensen}. The cohomology of virtually torsion-free solvable groups of finite rank.  To appear in {\it Trans. Amer. Math. Soc.}. Final version available on arxiv.org as  arXiv:1303.5005v4.

\bibitem[{\bf 4}]{robinson-lennox1}{\sc J. C. Lennox} and {\sc D. J. S. Robinson}. Soluble products of nilpotent groups. {\it Rend. Sem. Mat. Univ. Padova}~{\bf 62} (1980), 261-80. 

\bibitem[{\bf 5}]{robinson-lennox2}{\sc J. C. Lennox} and {\sc D. J. S. Robinson}. {\it The Theory of Infinite Soluble Groups} (Oxford, 2004).


\bibitem[{\bf 6}]{robinson}{\sc D. J. S. Robinson.} On the cohomology of soluble groups of finite rank. {\it J. Pure Appl. Algebra}~{\bf 43} (1978), 281-287



\bibitem[{\bf 7}]{wilson}{\sc J. S. Wilson.} {\it Profinite Groups.} (Oxford, 2002). 



\end{thebibliography}
\end{document}